\documentclass[11pt]{amsart}

\usepackage{amsfonts}
\usepackage{amssymb}
\usepackage{amsmath}
\usepackage{amsthm}
\usepackage{pdfpages}
\usepackage{enumerate}
\usepackage[all]{xy}
\usepackage{graphicx}
\usepackage{appendix}
\usepackage{color}

\newtheorem*{starthma}{Theorem A}
\newtheorem*{starthmb}{Theorem B}
\newtheorem{theorem}{Theorem}

\newtheorem{lemma}{Lemma}
\newtheorem{remark}{Remark}
\newtheorem{conjecture}{Conjecture}
\newtheorem{example}{Example}

\newtheorem{question}[theorem]{Question}
\theoremstyle{definition}
\newtheorem{definition}{Definition}
\newtheorem{problem}{Problem}

\def\m{\mu}
\def\n{\nu}

\def\T{\mathbb{T}}
\def\S{\Sigma}
\def\s{\sigma}
\def\B{\mathcal{B}}

\def\P{\mathcal{P}}

\def\N{\mathbb{N}}
\def\Si{\Sigma} 
\def\s{\sigma} 
\def\P{\mathcal{P}}

\title[Geometry in the Furstenberg Conjecture]{Geometry in the Furstenberg Conjecture}

\author{Yunping JIANG}
\address{Yunping Jiang: Department of Mathematics, 
Queens College of the City University of New York,
Flushing, NY 11367-1597 and 
Department of Mathematics
Graduate School of the City University of New York
365 Fifth Avenue, New York, NY 10016}
\email{yunping.jiang@qc.cuny.edu}
 


\begin{document}

\begin{abstract}
We explore the geometric aspects of the Furstenberg conjecture, proving that a non-atomic probability measure on the unit circle, invariant under both $p$- and $q$-actions for coprime integers $p,q>1$, must be the Lebesgue measure if it exhibits balanced geometry for one of these actions. Within rigidity theory, we show that balanced geometry is equivalent to the Lipschitz property. A consequence is that an orientation-preserving homeomorphism of the circle conjugating both $p$- and $q$-actions and preserving the Lebesgue measure must be the identity if one of these conjugations satisfies the Lipschitz property.
Our approach does not rely solely on ergodicity, and we conclude by proposing conjectures and open problems that frame the Furstenberg conjecture through geometric and quasisymmetric perspectives.
\end{abstract}

\maketitle


\section{The Statement of Main Result}~\label{bgsect}

Let $\mathbb{R}$ denote the real line and let $\mathbb{Z}$ denote the set of integers. We define $\mathbb{T}=\mathbb{R}/\mathbb{Z}$ as the unit circle and $\mathcal{B}$ as the standard Borel $\sigma$-field on $\mathbb{T}$. The Lebesgue probability measure on $\mathbb{T}$ is denoted by $\mu_{0}(A) =\int_{A} dx$ for $A\in \mathcal{B}$, where $dx$ denotes the Lebesgue metric on $\mathbb{R}$. Let $\mu$ be a non-atomic probability Borel measure on $\mathbb{T}$, referred to as a measure. 
Denote $I=[a,b]$ as an arc on $\mathbb{T}$ obtained by the interval $[a,b]$ in $\mathbb{R}$ modulo $1$ and $|I|=\mu_{0}(I)$ as the length of the arc. 

For $p>1\in \mathbb{N}$, consider the $p$-action on $\mathbb{T}$ defined as
$$
f_{p}: x\mapsto px \pmod{1}.
$$ 
The $p$-action induces a partition on $\mathbb{T}$ into intervals
 $$
\eta_{1}=\Big\{ I_{i}=\Big[\frac{i}{p}, \frac{i+1}{p}\Big]\;\Big|\; 0\leq i\leq p-1\Big\}.
 $$
The $n$-iteration $f^{n}_{p}$ induces a partition on $\mathbb{T}$ into intervals
$$
\eta_{n} =\{ I_{w_{n}}\;|\; w_{n}\in \Si_{n}\}, \quad n\in \N,
$$
where 
$$
\Si_{n}= \{w_{n}=i_{0}i_{1}\cdots i_{k}\cdots i_{n-1}\;|\; i_{k}\in \{0, 1, \cdots, p-1\}, 0\leq k\leq n-1\}
$$ 
denotes the set of all sequences of length $n$ over the alphabet ${0, 1, \cdots, p-1}$ and 
$
I_{w_{n}} =I_{i_{0}}\cap f^{-1}_{p} (I_{i_{1}})\cap \cdots \cap f^{-(n-1)}_{p} (I_{i_{n-1}}).
$

Consider the shift:
$$
\s : w_{n} \in \Si_{n}\mapsto \s(w_{n}) = i_{1}\cdots i_{n-1}\in \Si_{n-1}.
$$
Then
$$
f_{p}: I_{w_{n}}\in \eta_{n}\to I_{\s(w_{n})}\in \eta_{n-1}.
$$ 
For any $w_{n-1}=i_{0}i_{1}\cdots i_{n-2}\in \Si_{n-1}$, we have that 
$$
I_{w_{n-1}}= \cup_{i=0}^{p-1} I_{w_{n-1}i}\quad \hbox{and}\quad
f^{-1}_{p}(I_{w_{n-1}})=\cup_{i=0}^{p-1} I_{iw_{n-1}}.
$$

For a given measure $\mu$, define
$$
\alpha (n-1)= \min_{w_{n-1}\in \Si_{n-1}, 0\leq i\leq p-1} \frac{\sum_{0\leq j\not=i\leq p-1} \m (I_{j w_{n-1}})}{\m(I_{iw_{n-1}})}.
$$
and 
$$
\beta (n-1) = \max_{w_{n-1}\in \Si_{n-1}, 0\leq i\leq p-1} \frac{\mu (I_{w_{n-1}})}{\mu(I_{w_{n-1}i})}
$$
\medskip
\begin{definition}[Balanced Geometry]~\label{bgeo}
We say that a measure $\m$ exhibits {\em balanced geometry} if there is a constant $C>0$ such that
$$
\alpha (n-1)\geq C, \quad \forall\;\; n\geq 1.
$$
\end{definition}

\medskip
\begin{definition}[Bounded Geometry]~\label{bdgeo}
We say that a measure $\m$ exhibits {\em bounded geometry} if there is a constant $C>0$ such that
$$
\beta (n-1) \leq C, \quad \forall\;\; n\geq 1.
$$
\end{definition}

\medskip
\begin{definition}[Quasisymmetry]~\label{qscc}
We say that a measure $\m$ is {\em quasisymmetric} 
if there is a constant $M\geq 1$ such that 
$$
M^{-1} \leq \frac{\mu([x, x+t])}{\mu([x-t,x])} \leq M, \quad \forall\;\; x\in \mathbb{T}, \;\; \forall \;t>0.
$$
\end{definition}

\medskip
\begin{remark}
The quasisymmetric condition in Definition~\ref{qscc} is stronger than the bounded geometry condition in Definition~\ref{bdgeo}; if $\mu$ is quasisymmetric, then it exhibits bounded geometry.   
\end{remark}

%

\medskip
\begin{definition}
A measure $\m$ is said to be $p$-invariant if
$$
\m( f^{-1}_{p}(A)) =\m (A), \quad \forall\;A\in \B.
$$
\end{definition}

The Lebesgue measure $\m_{0}$ is $p$-invariant for any $p>1$ and exhibits balanced geometry. We are interested in the Furstenburg conjecture as follows.

\medskip
\begin{conjecture}~\label{fcc}
If a measure $\m$ is both $p$- and $q$-invariant for a coprime pair of integers $p, q>1$ and {\em ergodic} with respect to the semigroup $<f_{p}, f_{q}>$ generated by $f_{p}$ and $f_{q}$, then it must be the Lebesgue measure $\mu_{0}$.
\end{conjecture} 

\medskip
\begin{remark} 
In the conjecture, $\mu$ is ergodic with respect to the semigroup $<f_{p}, f_{q}>$ generated by $f_{p}$ and $f_{q}$ means that for any $A\in \mathcal{B}$ with $f_{p}^{-1}(A)=A$ and $f_{q}^{-1}(A) =A$, then $\mu (A)=0$ or $1$.
\end{remark}

With an additional condition on metric entropy, Rudolph~\cite{R} proved that.

\begin{starthma} 
If a measure $\m$ is both $p$- and $q$-invariant for a coprime pair of integers $p, q>1$ and {\em ergodic} with respect to the semigroup $<f_{p}, f_{q}>$ generated by $f_{p}$ and $f_{q}$ and the metric entropy $h_{\mu}(f_{p})>0$, then it must be the Lebesgue measure $\mu_{0}$.
\end{starthma}

With a stronger condition on ergodicity and conservative, Host~\cite{H} proved that

\begin{starthmb} 
If a measure $\m$ is both $p$- and $q$-invariant for a coprime pair of integers $p, q>1$ and {\em ergodic} and {\em conservative} concerning $f_{p}$, then it must be the Lebesgue measure $\mu_{0}$.
\end{starthmb}

This paper will prove the following theorem under the balanced geometry without assuming ergodicity. 

\medskip
\begin{theorem}[Main Theorem]~\label{main}
If a measure $\m$ is both $p$- and $q$-invariant for a coprime pair of integers $p, q>1$ and exhibits balanced geometry for one of these actions, it must be the Lebesgue measure $\mu_{0}$.
\end{theorem} 

We are still working on the following conjecture under the bounded geometry without assuming ergodicity (see~\cite{AJ}).

\medskip
\begin{conjecture}~\label{bgconj}
If a measure $\m$ is both $p$- and $q$-invariant for a coprime pair of integers $p, q>1$ and exhibits bounded geometry for one of these actions, it must be the Lebesgue measure $\mu_{0}$.
\end{conjecture}

\medskip

The following conjecture plus Theorem A implies Conjecture~\ref{fcc}. 
\medskip
\begin{conjecture}~\label{qsconj}
If a measure $\m$ is both $p$- and $q$-invariant for a coprime pair of integers $p, q>1$ and ergodic concerning the semigroup $<f_{p}, f_{q}>$ generated by $f_{p}$ and $f_{q}$, then it is quasisymmetric.
\end{conjecture}

However, we believe that the following problem is reasonable. 

\medskip
\begin{problem}~\label{cexample}
Construct a measure which is $p$- and $q$-invariant for a coprime pair of integers $p, q>1$ and which is non-atomic and non-ergodic concerning the semigroup $<f_{p}, f_{q}>$ generated by $f_{p}$ and $f_{q}$. 
\end{problem}

\medskip
\begin{remark}
The measure in Problem~\ref{cexample} should not be quasisymmetric. Thus, it must exhibit some parabolic phenomenon like the one we will construct in Example~\ref{nex}. 
\end{remark}

The structure of the paper is as follows: Section~\ref{prsect} offers a detailed proof of the main result (Theorem~\ref{main}). We show the equivalence between the balanced geometry for a measure and the Lipschitz property on an orientation-preserving circle covering in Section~\ref{rsect}. Leveraging this equivalence, Section~\ref{rsect} establishes another result (Theorem~\ref{main2}) confirming Question~\ref{q}. Finally, Section~\ref{rsect} presents a non-rigidity problem when it lacks ergodicity and restates the Furstenberg conjecture as a conjecture in rigidity (Conjecture~\ref{fc}), breaking it into two conjectures involving quasisymmetry (Conjecture~\ref{qsc}) and bounded geometry (Conjecture~\ref{rbgc}) on a homeomorphism of the circle.
We present in Section~\ref{enx} an example of a measure exhibiting balanced geometry and a counter-example of a measure exhibiting imbalanced geometry, which will be helpful for us to study further Problem~\ref{p} and Conjecture~\ref{qsc}. Two appendixes (Appendix I and Appendix II) will provide detailed proof of a result in number theory and state the martingale convergence theorem, both of which will be used to prove the main result. 


\section{\bf The Proof of the Main Result (Theorem~\ref{main})}~\label{prsect}

 Suppose $\mu$ exhibits balanced geometry for the $p$-action. Let $\nu_{0}=\m$ and, for any $n>0$, define 
$$
\nu_{n} (A) = \frac{1}{p^{n}}\sum_{k=0}^{p^{n}-1} \m \Big(A+ \frac{k}{p^{n}}\Big), \quad A\in \B, \;\;n\geq 1,
$$
which is a measure. 

\medskip
\begin{remark}
Consider the Ruelle operator associated with the dynamical system $f_{p}(x)$ and the potential $\psi (x) =-\log f_{p} (x)$,
$$
{\mathcal L} \phi (x) =\sum_{y\in f_{p}^{-1}(x)} e^{\psi (y)} \phi (y): C(\T) \to C(\T),
$$
where $C(\T)$ means the space of all continuous functions on $\T$. Since ${\mathcal L}1=1$, its dual operator
${\mathcal L}^{*}$ maps the space $M_{1} (\T)$ of all probability measures on $\T$ into itself. We have that 
$$
{\mathcal L}^{n} \phi (f^{n}_{p}x)  = \frac{1}{p^{n} } \sum_{k=0}^{p^{n}-1}  \phi \Big(x+\frac{k}{p^{n}}\Big)
$$
whose dual operator $({\mathcal L}^{n}\circ f^{n}_{p})^{*}$ mapping $M_{1} (\T)$ into itself.  Then
$$
\nu_{n} =  ({\mathcal L}^{n} \circ f^{n}_{p})^{*} \mu, \quad \forall\; n>1.
$$
\end{remark}

 Let $g_{i}: [0,1]\to I_{i}$, $0\leq i\leq p-1$, be the inverse branches of $f_{p}$. For any $w_{n}=i_{0} i_{1}\cdots i_{n-1}\in\Si_{n}$, let
$
g_{w_{n}}= g_{i_{0}}\circ \cdots \circ g_{i_{n-1}}.
$
Then for $k=\sum_{j=0}^{n-1} i_{j}p^{j}$,
$$
g_{w_{n}} (x) =\frac{x}{p^{n}} +\frac{k}{p^{n}} \pmod{1}.
$$
When $A$ is contained in an interval in $\eta_{n}$, we have
$$
\nu_{n} (A) = \frac{1}{p^{n}} \sum_{w_{n}\in \Si_{n}} \m (g_{w_{n}} (f^{n}_{p} (A)))= \frac{1}{p^{n}} \m (f_{p}^{-n} (f^{n}_{p} (A)))=\frac{1}{p^{n}} \m (f^{n}_{p} (A))
$$
since $\m$ is $p$-invariant. 
Since $\nu_{n-1} (A) \leq p \nu_{n} (A)$ for all $n\geq 1$, $\nu_{n-1}$ is absolutely continuous with respect to $\nu_{n}$, thus we have the Radon-Nikodym derivative
$
\phi_{n} =d\m/d\nu_{n}.
$

For $\nu_{n}$-a.e. $x_{w}$ where $w=w_{k}\cdots =i_{0}i_{1}\cdots i_{k-1}\cdots \in \S$ and for any $k>n>m\geq 0$, we have that
$$
\frac{\nu_{m} (I_{w_{k}})}{\nu_{m+1} (I_{w_{k}})}= \frac{\frac{1}{p^{m}} \m (f_{p}^{-m}(f_{p}^m(I_{w_{k}})))}{\frac{1}{p^{m+1}} \m (f_{p}^{-(m+1)}(f_{p}^{(m+1)}(I_{w_{k}})))}
$$
$$
=\frac{p\m (I_{\s^{m}(w_{k})})}{\m (f_{p}^{-1}(f_{p}(I_{\s^{m}(w_{k})})))}
=\frac{p \m (I_{\s^{m}(w_{k})})}{\sum_{j=0}^{p-1} \m (I_{j\s^{m+1} (w_{k})})}
$$
$$
=\frac{p}{1+\frac{\sum_{0\leq j\leq p-1, j\s^{m+1}(w_{k}) \not= \s^{m}(w_{k})} \m (I_{j\s^{m+1}(w_{k})})}{\m (I_{\s^{m}(w_{k})})}}
\leq \frac{p}{1+\alpha (k-m-1)}.
$$
Therefore, we get 
$$
\frac{\m (I_{w_{k}})}{\nu_{n} (I_{w_{k}})} =\prod_{m=0}^{n-1} \frac{\nu_{m} (I_{w_{k}})}{\nu_{m+1} (I_{w_{k}})}\leq \frac{p^{n}}{\prod_{m=0}^{n-1}(1+\alpha (k-m-1)} \leq \frac{p^{n}}{(1+C)^{n}},
$$
where $C>0$ is the constant in the definition of balanced geometry,
and
\begin{equation}~\label{bd}
\phi_{n} (x) =\lim_{k\to \infty} \frac{\m(I_{w_{k}})}{\nu_{n} (I_{w_{k}})} \leq \frac{p^{n}}{(1+C)^{n}} \quad \hbox{for $\nu_{n}-a.e. \; x\in \mathbb{T}$}
\end{equation}
(refer to Lemma~\ref{mct} in Appendix II).

Let $B_{n}=\{ x\in \T\;|\; \phi_{n}(x) \hbox{ is not defined}\}$. We have $\m(B_{n})\leq \nu_{n} (B_{n})=0$, Consider 
$$
A_{n} = \{ x\in \T\setminus B_{n}\;|\; \phi_{n}(x) =0\} =\phi_{n}^{-1} (0).
$$
It is measurable, and 
$$
\m (A_{n}) =\int_{A_{n}} \; d\m (x) = \int_{A_{n}} \phi_{n}(x) \; d\n_n (x)=0.
$$
Let  
$$
K=\{x\;|\; \phi_n(x)>0\mbox{ for all }n\geq1\} =\bigcap_{n\geq 1} (A_{n}\cup B_{n})^{c} = \Big( \bigcup_{n=1}^{\infty} (A_{n}\cup B_{n}) \Big)^{c}
$$
Then we have that $\m (K)=1$. 
On $K$, we have 
$$
d\n_n(x)=\frac{1}{\phi_n(x)}\;d\m (x), \quad \forall\; n\geq 1.
$$

The measure $\mu$ is the Lebesgue measure $\mu_{0}$ if and only if its Fourier coefficient $f_{m}=\int_{\T} e^{2\pi i m x} d\m (x)=0$ for all integers $m\geq 1$. Denote $e(x) =e^{2\pi i x}$ and let us fix an integer $m\geq 1$. Let $T_{n}$ and $C_{1}$ and $n_{0}$ be the numbers in Lemma~\ref{nt} in Appendix I.

For each $n>n_{0}$, we define
$$
\beta_{T_{n}} (x)=\frac{1}{T_{n}}\sum_{k=0}^{T_{n}-1}e(mq^kx).
$$
Then 
$$
 \int_{K}\frac{|\beta_{T_{n}}(x)|^2}{\phi_n(x)}\,d\mu(x) = \int_{K} |\beta_{T_{n}}(x)|^2\,d\nu_n(x)
$$
$$
=\frac{1}{p^{n}} \int_{K} \sum_{k=0}^{p^{n}-1} \Big| \beta_{T_{n}}\Big(x+\frac{k}{p^{n}}\Big)\Big|^2\,d\mu(x)=\int_{K} h(x)\,d\mu(x)=\int_{\T}h(x)\,d\m(x),
$$
where
$$
h(x)=\frac{1}{p^{n}} \frac{1}{T_{n}^2}\sum_{k=0}^{T_{n}-1}\sum_{l=0}^{T_{n}-1}e(m(q^k-q^l)x)\sum_{j=0}^{p^n-1}e\Big(m\Big(\frac{(q^k-p^l)j}{p^{n}}\Big) \Big).
$$
If $0\leq k\neq l\leq T_{n}$, then $q^k\neq q^l\mod p^n$, which implies that 
$$
\sum_{j=0}^{p^n-1}e\Big(m\Big( \frac{(q^k-q^l)j}{p^{n}}\Big)\Big)=0.
$$
So we have $h(x)=1/T_{n}$, and $\int_{\T} h(x)d\mu(x)=1/T_{n}$, which  implies that 
$$
\int_{K}\frac{|\beta_{T_{n}}(x)|^2}{\phi_n(x)}\,d\mu(x) = \frac{1}{T_{n}}.
$$
Using the Cauchy-Schwarz inequality, 
$$
\left|\int_{\T} \beta_{T_n}(x)\,d\mu(x)\right|^2=\left|\int_{K} \beta_{T_n}(x)\,d\mu(x)\right|^2 \leq\int_{K}\frac{|\beta_{T_n}(x)|^2}{\phi_n(x)}\,d\mu(x)\cdot\int_{K}\phi_n(x)\,d\mu(x)
$$
$$
=\frac{1}{T_n}\int_{K}\phi_n(x)\,d\mu(x) =\frac{1}{C_{1}p^{n}} \int_{K}\phi_n(x)\,d\mu(x)=\frac{1}{C_{1}p^{n}} \int_{\T}\phi_n(x)\,d\mu(x).
$$
From (\ref{bd}), we have 
$$
\Big| \int_{\T} \beta_n(x)\,d\mu(x)\Big| \leq \frac{1}{\sqrt{C_{1}}(1+C)^{\frac{n}{2}}}\to 0 \quad \hbox{as} \quad n\to \infty.
$$
Since $\m$ is $q$-invariant, 
$$
\int_{\T} \beta_{T_n}(x)\,d\mu(x)=\int_\T e(mx)\,d\mu(x)=f_{m}, \;\; \forall \; n\geq n_{0}.
$$ 
We get $f_{m}=0$, 
which completes the proof of the theorem.

\medskip
\begin{remark}
We do not need the ergodicity assumption on $\m$ in the proof, while the proofs in~\cite{R,J,H} reply to this assumption. 
The balanced geometry and the ergodicity on $\m$ imply the positivity of the metric entropy. See other interesting discussions  in~\cite[Chapter 9]{ELW} and~\cite{AA}).
 \end{remark}
 
 \section{\bf Rigidity}~\label{rsect}

Let $h$ be an orientation-preserving homeomorphism of $\mathbb{T}$ with $h(0)=0$ and $p>1\in \mathbb{N}$. Define
\[
F_{p}= h\circ f_{p}\circ h^{-1}: \mathbb{T}\to \mathbb{T}.
\]
It is a degree $p$ topological covering of $\mathbb{T}$. Let 
\[
\{ G_{i} =(F|h(I_{i}))^{-1}: [0,1] \to h(I_{i})\;|\; 0\leq i\leq p-1\}
\] 
be the set of inverse branches of $F_{p}$. We say $F_{p}$ preserves the Lebesgue measure if $(F_{p})_{*}\mu_{0}=\mu_{0}$, equivalently, 
\[
\mu_{0} (F_{p}^{-1}(A)) =\mu_{0} (A), \quad \forall A\in \mathcal{B}.
\]
We say that $F_{p}$ has the Lipschitz property if all $G_{i}$ are Lipschitz functions with the Lipschitz constant less than $1$ for $0\leq i\leq p-1$. In his recent presentation at Stony Brook~\cite{S}, Sullivan posed a query in rigidity. 

\medskip
\begin{question}~\label{q}
Suppose $F_{p}$ and $F_{q}$ both preserve the Lebesgue measure for a coprime pair of integers $p>1$ and $q>1$, and one of them has the Lipschitz property. Is $h$ the identity map? 
\end{question}

This paper also provides a positive response to this inquiry.
Using Theorem~\ref{main}, we give an affirmative answer to the question.

\medskip
\begin{theorem}\label{main2}
Suppose $F_{p}$ and $F_{q}$ both preserve the Lebesgue measure for a coprime pair of integers $p>1$ and $q>1$, and one of them has the Lipschitz property. Then $h$ must be the identity map.
\end{theorem}

\begin{proof}
Consider the push-forward measure $\mu=(h^{-1})_{*} \mu_{0}$. Then $\mu$ is $p$- and $q$-invariant since for any $A\in \mathcal{B}$,
\[
\mu (f_{p}^{-1}(A)) = \mu_{0} (h (f_{p}^{-1} (h^{-1} (h(A)))=\mu_{0}(F_{p}^{-1} (h(A))) =\mu_{0} (h(A)) =\mu (A)
\]
and 
\[
\mu (f_{q}^{-1}(A)) = \mu_{0} (h (f_{q}^{-1} (h^{-1} (h(A)))=\mu_{0}(F_{q}^{-1} (h(A))) =\mu_{0} (h(A)) =\mu (A)
\]

Suppose $F_{p}$ has the Lipschitz property. We will check that $\mu$ exhibits balanced geometry concerning the $p$-action. Let $L<1$ be the maximum of the Lipschitz constants of $G_{i}$ for $0\leq i\leq p-1$. 
For any $w_{n-1}=i_{0}i_{1}\cdots i_{n-2}\in \Sigma_{n-1}$ and any $0\leq i\leq p-1$, since  
\[
f^{-1}_{p}(I_{w_{n-1}})=\bigcup_{j=0}^{p-1} I_{jw_{n-1}},
\]
we get 
\[
\sum_{j=0}^{p-1} \mu (I_{jw_{n-1}}) = \mu (I_{w_{n-1}}) =\mu (f_{p}(I_{iw_{n-1}})),
\]
which implies that
\[
1+  \frac{\sum_{0\leq j\not=i\leq p-1} \mu (I_{j w_{n-1}})}{\mu(I_{iw_{n-1}})}=\frac{\mu (f_{p}(I_{iw_{n-1}}))}{\mu(I_{iw_{n-1}})}=
\frac{|h(I_{w_{n-1}})|}{|G_{i}(h(I_{w_{n-1}}))|} \geq \frac{1}{L}.
\]
Thus,
\[
\frac{\sum_{0\leq j\not=i\leq p-1} \mu (I_{j w_{n-1}})}{\mu(I_{iw_{n-1}})}\geq C_{0}=\frac{1}{L}-1>0.
\]
 Now Theorem~\ref{main} implies that $\mu=\mu_{0}$ and $h=\text{Id}$. It completes the proof.
\end{proof}

A problem in rigidity similar to Problem~\ref{cexample} is the following problem.

\medskip
\begin{problem}~\label{p}
Construct a non-trivial homeomorphism $h$ of $\mathbb{T}$ with $h(0)=0$ such that both $F_{p}$ and $F_{q}$ preserve the Lebesgue measure for a coprime pair of integers $p, q>1$ and do not have the Lipschitz property. 
\end{problem} 
 
The Furstenburg conjecture (Conjecture~\ref{fcc}) can be a rigidity conjecture in one-dimensional dynamics. In~\cite{F}, Furstenburg showed that any closed set $S\subseteq \mathbb{T}$ is invariant under multiplication by a non-lacunary semigroup of integers must be either finite or the whole $\mathbb{T}$. Thus, if $\mu$ is $p$- and $q$-invariant measure on $\mathbb{T}$ for a coprime pair of integers $p>1$ and $q>1$, then it is either atomic or the support $supp (\mu)=\mathbb{T}$.  Therefore, if, in addition,  $\mu$ is non-atomic, then $h(x) =\mu ([0,x])$ defines a homeomorphism of $\mathbb{T}$ such that both $F_{p}$ and $F_{q}$ preserve the Lebesgue measure $\mu_{0}$.  Let $< f_{p}, f_{q}>$ be the non-lacunary semi-group generated by $f_{p}$ and $f_{q}$. It is Abelian since 
$$
(f_{p}\circ f_{q}) (x)=(f_{q}\circ f_{p}) (x) =pq x \pmod{1}, \quad x\in \mathbb{T}. 
$$ 
Then $< F_{p}, F_{q}>$ generated by $F_{p}$ and $F_{q}$ is an Abelian semi-group acting on $\mathbb{T}$. We say $\mu$ is $<f_{p}, f_{q}>$ ergodic if it is an ergodic action on the probability space $(\mathbb{T}, \mathcal{B}, \mu)$. It is equivalent to say that $<F_{p}, F_{q}>$ is an ergodic action on the probability space $(\mathbb{T}, \mathcal{B}, \mu_{0})$, in other words, $<F_{p}, F_{q}>$ is ergodic (respective to the Lebesgure measure). 
We have an equivalent statement to the Fursenburg conjecture (Conjecture~\ref{fcc}) in rigidity. 

\medskip
\begin{conjecture}~\label{fc}
 Suppose $F_{p}$ and $F_{q}$ both preserve the Lebesgue measure and $<F_{p}, F_{q}>$ is ergodic for a coprime pair of integers $p>1$ and $q>1$. Then $h$ must be the identity map.
\end{conjecture}
  
Our initial investigation~\cite{AJ}  in this direction is to study rigidity under bounded geometry following our symmetric rigidity work in~\cite{AHJW} under bounded geometry. However, we need to work more on the proof. We list it here as a conjecture. We say $F_{p}$ exhibits bounded geometry if there is a constant $B>0$ such that, for any $w_{n-1} =i_{0}\cdots i_{n-1}\in \Sigma_{n-1}$ and any $0\leq i\leq p-1$, 
$$
\frac{|h(I_{w_{n-1}})|}{|h(I_{w_{n-1}i})|} \leq B.
$$
 
 \medskip
\begin{conjecture}~\label{rbgc}
 Suppose $F_{p}$ and $F_{q}$ both preserve the Lebesgue measure for a coprime pair of integers $p>1$ and $q>1$, and $F_{p}$ exhibits bounded geometry. Then $h$ must be the identity map.
\end{conjecture}

Conjecture~\ref{rbgc} combined with the following conjecture implies Conjecture~\ref{fc}.

\medskip
\begin{conjecture}~\label{qsc}
Suppose $F_{p}$ and $F_{q}$ both preserve the Lebesgue measure 
and $<F_{p},F_{q}>$ is ergodic for a coprime pair of integers $p>1$ 
and $q>1$. Then $h$ is a quasisymmetric homeomorphism. 
\end{conjecture}

In Conjecture~\ref{qsc}, $h$ is a quasisymmetric homeomorphism, meaning that there is a constant $M\geq 1$ so that 
$$
\frac{1}{M} \leq \frac{H(x+t)-H(x)}{H(x)-H(x-t)} \leq M, \quad \forall x\in {\mathbb R}, \;\;\forall\; t>0,
$$
where $H$ is the lift homeomorphism of $h$ to ${\mathbb R}$ such that $H(0)=0$. We have that $H(x+1)=H(x)+1$ for all $x\in {\mathbb R}$. The quasisymmetric property implies bounded geometry.



 \section{\bf Examples and Counter-Example in Balanced Geometry}~\label{enx}
 
We give a non-trivial example of a $p$-invariant measure having balanced geometry (Example~\ref{pex}) and an example of $p$-invariant measure having imbalanced geometry (Example~\ref{nex}). These two examples show that the balanced geometry condition is interesting regarding rigidity.

Suppose $f: \T\to \T$ is a covering map of degree $p$. We call it a circle endomorphism. We say $f$ is $C^{1}$ if it has a continuous derivative $f'$. We say a $C^{1}$ circle endomorphism $f$ expanding if there are two constants $C>0$ and $\lambda >1$ such that
$$
|(f^{n})' (x)| \geq C \lambda^{n} \quad x\in \T.
$$
We say $f$ is $C^{1+\alpha}$ if $f'$ is an $\alpha$-H\"older continuous function for some $0<\alpha\leq 1$, that is, there are two constants $0<a <1$ and $D>0$ such that 
$$
|f'(x) -f'(y)| \leq D|x-y|^{\alpha} \quad x, y \in \T \hbox{ with } |x-y|<a.
$$
Any $C^{1+\alpha}$ expanding endomorphism $f$ has a unique smooth $f$-invariant probability measure $\nu$ (refer to~\cite{JRue,JRig}). This means that we have a positive $\alpha$-H\"older continuous function $\rho (x)$ on $\T$ such that
$$
\nu (A) =\int_{A} \rho (x) dx\quad \hbox{and}\quad \nu (f^{-1} (A)) =\nu (A), \quad \forall A\in \B.
$$
The formula $h_{1}(x)=\int_{0}^{x} \rho(x) dx$ defines a $C^{1+\alpha}$-diffeomorphism of $\T$. One can check that 
$$
g= h_{1} \circ f\circ h_{1}^{-1}
$$
preserves the Lebesgue measure $\m_{0}$.
Furthermore, $g$ is topologically conjugate to $f_{p}$, that is, we have a homeomorphism $h: \T\to \T$ such that
$$
g\circ h= h\circ f_{p}
$$
Let $\m=(h^{-1})_{*} \m_{0}$ defined as $\m(A) =\m_{0} (h(A))$ for all $A\in \B$. In general, $\m$ is totally singular with respective to $\m_{0}$ since, otherwise, $g=f_{p}$ (see~\cite{JRig}). 
Actually, in general, $h\not=id$ is a quasisymmetric homomorphism, and as long as $h$ is symmetric, then $h=id$ (see~\cite{JRig}).  

 \medskip
\begin{example}~\label{pex}
The measure $\m$ constructed above is $p$-invariant and has balanced geometry.
\end{example}

\begin{proof}
For $p$-invariance, it is because that for any $A\in \B$, 
$$
\m(f_{p}^{-1}(A)) =\m_{0} (h(f_{p}^{-1}(A))) =\m_{0} (g^{-1} (h(A))) = \m_{0} (h(A))=\m (A).
$$
 
For balanced geometry, without loss of generality, we assume that
$$
1< m =\min_{x\in \T} |g'(x)| \leq M=\max_{x\in \T} |g'(x)| <\infty.
$$ 
For any $n\geq 2$ and $w_{n-1}\in \S_{n-1}$ and $0\leq i\leq p-1$ , we have that 
$$
f_{p}^{-1} (I_{w_{n-1}}) = I_{iw_{n-1}} \cup \cup_{0\leq j\not=i\leq p-1} I_{jw_{n-1}}.
$$
and 
$$
g^{-1} (h(I_{w_{n}})) =h(I_{iw_{n-1}}) \cup \cup_{0\leq j\not=i\leq p-1} h(I_{jw_{n-1}}).
$$
Since $g$ preserves the Lebesgue measure $\m_{0}$, we have that
$$
|h(I_{w_{n-1}})| =|g^{-1} (h(I_{w_{n}}))| = |h(I_{iw_{n-1}})| + \sum_{0\leq j\not=i\leq p-1} |h(I_{jw_{n-1}})|.
$$
Since $g(h(I_{iw_{n-1}})) = h(I_{w_{n-1}})$, we have that, for any $j\not=i$,
$$
\frac{|g(h(I_{iw_{n-1}}))|}{|h(I_{iw_{n-1}})|} =1+ \frac{|h(I_{jw_{n-1}})|}{|h(I_{iw_{n-1}})|} +\sum_{0\leq k\not=i, j\leq p-1} 
\frac{|h(I_{kw_{n-1}})|}{|h(I_{iw_{n-1}})|},
$$
which implies that
$$
 \frac{|h(I_{jw_{n-1}})|}{|h(I_{iw_{n-1}})|} \Big( 1+\sum_{0\leq k\not=i, j\leq p-1} 
\frac{|h(I_{kw_{n-1}})|}{|h(I_{jw_{n-1}})|} \Big) 
= \frac{|g(h(I_{iw_{n-1}}))|}{|h(I_{iw_{n-1}})|} -1 \geq m -1.
$$

Since each 
$$
\frac{|h(I_{kw_{n-1}})|}{|h(I_{jw_{n-1}})|}
=
\frac{\frac{|h(I_{w_{n-1}})|}{|h(I_{jw_{n-1}})|}}
{\frac{|h(I_{w_{n-1}})|}{|h(I_{jw_{n-1}})}|}
=
\frac{\frac{|g(h(I_{kw_{n-1}}))|}{|h(I_{jw_{n-1}})|}}
{\frac{|g(h(I_{iw_{n-1}}))|}{|h(I_{jw_{n-1}})}|} \leq \frac{M}{m},
$$
$$
 \frac{\mu (I_{jw_{n-1}})}{\mu (I_{iw_{n-1}})} = \frac{|h(I_{jw_{n-1}})|}{|h(I_{iw_{n-1}})|} \geq C_{0}=\frac{m-1}{1+(p-2)\frac{M}{m}}>0,
 $$
which says that $\m$ has balanced geometry. 

\end{proof}

Next, we construct a $p$-invariant measure having imbalanced geometry. 
For a given $0<\alpha<1$, let $f$ be a $C^{1+\alpha}$ almost expanding circle endomorphism of degree $2$, that is, $f(0)=0$ and $f'(0)=1$ and $|f'(x)|>1$ for all $x\not= 0\in \T$. Assume $f(x) =x +x^{1+\alpha} +o(x^{1+\alpha})$ in a small neighborhood of $0$. Then $f$ has an $f$-invariant probability measure $\m$ absolutely continuous with respect to the Lebesgue measure $\m_{0}$ such that the density $\rho (x) =d\m/d\m_{0} (x)$ is continuous on $(0,1]$ but $\rho (x) \to \infty$ as $x \to 0$ (refer to~\cite{Hu}). 
The formula $h_{1}(x)=\int_{0}^{x} \rho(x) dx$ defines a homeomorphism of $\T$ and $C^{1}$ on $(0,1)$. The map 
$$
g= h_{1} \circ f\circ h_{1}^{-1}
$$
preserves the Lebesgue measure $\m_{0}$ and $g(0)=0$ and $g'(0)=1$. It is topologically conjugate to $f_{p}$. That is, there is a homeomorphism $h: \T\to \T$ such that
$$
g\circ h= h\circ f_{p}
$$
Let $\m=(h^{-1})_{*} \m_{0}$ defined as $\m(A) =\m_{0} (h(A))$ for all $A\in \B$.

\medskip
\begin{example}~\label{nex}
The measure $\m$ is $p$-invariant and has imbalanced geometry.
\end{example}

\begin{proof}
For $p$-invariance, it is because that for any $A\in \B$, 
$$
\m(f_{p}^{-1}(A)) =\m_{0} (h(f_{p}^{-1}(A))) =\m_{0} (g^{-1} (h(A))) = \m_{0} (h(A))=\m (A).
$$
 
For any $n\geq 2$ and $0_{n-1}=\underbrace{0\cdots 0}_{n-1} \in \S_{n-1}$ and $i=0$ , we have that 
$$
f_{p}^{-1} (I_{0_{n-1}}) = I_{0_{n}} \cup \cup_{0<i\leq p-1} I_{i0_{n-1}}.
$$
and 
$$
g^{-1} (h(I_{0_{n-1}})) =h(I_{0_{n}}) \cup h(\cup_{0<i\leq p-1}I_{i0_{n-1}}).
$$
Since $g$ preserves the Lebesgue measure $\m_{0}$, we have that
$$
|h(I_{0_{n-1}})| =|g^{-1} (h(I_{0_{n}}))| = |h(I_{0_{n}})| + \sum_{i=1}^{p-1} |h(I_{i0_{n-1}})|
$$
and that 
$$
 \frac{|h(I_{0_{n-1}})|}{|h(I_{0_{n}})|} = 1+
\frac{\sum_{i=1}^{p-1} |h(I_{i0_{n-1}})|}{|h(I_{0_{n}})|} 
$$
Since $g(h(I_{0_{n}}))=h(I_{0_{n-1}})$ and $g'(0)=1$, 
$$
\frac{\sum_{i=1}^{p-1} |h(I_{i0_{n-1}})|}{|h(I_{0_{n}})|} \to 0 \quad \hbox{as} \quad n\to \infty,
$$  
which says that $\mu$ has imbalanced geometry. 
 \end{proof}

\section{\bf Appendix I: A Result in Number Theory}~\label{rnt}

In Appendix I, we give a detailed proof of a result in number theory (Lemma~\ref{nt}). We used the result in the proof of Theorem~\ref{main}. Another way to see this result is from the Carmichael function $\lambda(n)$ in number theory. 

\medskip
\begin{lemma}~\label{nt} 
Suppose $p>1$ and $q>1$ are two relatively prime integers. Suppose $a\geq 1$ is an integer. 
Let 
$$
T_{n}=\#\Big(\{ aq^k \pmod{p^n}\;\;|\; k=0, 1, \cdots\}\Big)
$$
Then there is a constant $C_{1}=C_{1}(a,p,q)>0$ and a constant integer $n_{0}=n_{0}(a, p, q)$ such that 
$$
T_n=C_{1}p^n, \;\;\forall \; n\geq n_0.
$$
\end{lemma}

\begin{proof}
Let $a=bp^r$ with $\gcd(b,p)=1$. For every $n\geq r+1$, consider two integers $k>l$ 
such that $aq^k=aq^l \pmod{p^n}$. That is, 
$$
aq^l(q^{k-l}-1)= m p^n
$$
for some integer $m\geq 0$, 
which implies that 
$$
b q^l(q^{k-l}-1)=mp^{n-r}
$$
Since $b$ and $q$ are both relatively prime to $p$, we have $p^{n-r}|(q^{k-l}-1)$, that is,
$$
q^{k-l}=1\pmod{p^{n-r}}.
$$
Thus $T_{n}$ is the smallest positive integer such that $q^{T_{n}}=1\mod p^{n-r}$. 

Suppose $q^{T_n}=k_{n} p^{n-r}+1$ for some integer $k_{n}\geq 1$. 
For $n=r+1$, we have
$q^{T_{r+1}}=k_{r+1}p+1$.
Let us write $k_{r+1}=u_0p^s$, where $(u_0,p)=1$. Then
$$
q^{T_{r+1}}=u_0p^{s+1}+1=u_0p^{(r+s+1)-r}+1=q^{T_{r+s+1}},
$$
which implies
$$
q^{pT_{r+s+1}}=(u_0p^{s+1}+1)^p
$$
$$
=\sum_{j=0}^p\left(\begin{array}{l}p\cr                           j
\end{array}\right)(u_0p^{s+1})^{p-j} =(u_0p^{s+1})^p+p(u_0p^{s+1})^{p-1}+\ldots+p(u_0p^{s+1})+1
$$
$$
=u_1p^{s+2}+1.
$$
Note that
$u_1=u_0+p(\cdots)=u_0\pmod{p}$.
Since $(u_0,p)=1$, we have $(u_1,p)=1$ and thus
$$
u_1p^{s+2}+1=q^{T_{r+s+2}}.
$$
Therefore,
$$
pT_{r+s+1}=T_{r+s+2}.
$$
Inductively we get
$$
T_{r+s+1+m}=p^mT_{r+s+1}, \;\;\forall \; m\geq 0.
$$
Setting $n_0=r+s+1$ we have
$$
T_{n_0+m}=p^mT_{n_0}, \;\;\forall \; m\geq 0.
$$
That is,
$$
\frac{T_{n_0+m}}{p^{n_0+m}}=\frac{T_{n_0}}{p^{n_0}},\;\;\forall m\geq 0.
$$
Therefore, let
$$
C_{1}=\frac{T_{n_0}}{p^{n_0}}.
$$
Then
\[T_n=C_{1}p^n, \quad\forall n\geq n_0.\]
\end{proof}
 
 \section{\bf Appendix II: Martingale Convergence Theorem}~\label{mctsect} 
 
Appendix II states the martingale convergence theorem (Lemma~\ref{mct}). We used it in the proof of Theorem~\ref{main}.

Consider $\P=(\T, \B, \nu_{n})$ as a probability space where $\nu_{n}$ is the probability measure constructed in the proof of Theorem~\ref{main}. Consider $X=\phi_{n}$ as a random variable on $\P$. Let $\B_{k}$ be the $\s$-field generated 
by $\eta_{k} =\{ I_{w_{k}} | w_{k}\in \S_{k}\}$ for every $k\geq 1$. Then we have a sequence of increasing $\s$-fields 
$$
\B_{1} \subset \B_{2}\subset \cdots \subset \B_{k}\subset \B_{k+1}\subset \cdots\subset \B
$$
and $\B=\cup_{k=1}^{\infty} \B_{k}$. Let $X_{k} =E(X|\B_{k})$ be the expectation of $X$ on $\B_{k}$, which is a $\B_{k}$-measurable function. Then we have $X_{k} = E(X_{k+1}| \B_{k})$, $\nu_{n}$-a.e.. Thus, 
$
\{ (X_{k}, \B_{k})\}_{k=1}^{\infty}
$
gives us a martingale on $\P$ and 
$$
X_{k} (x) =\frac{ \int_{I_{w_{k}}} Xd\nu_{n}} { \nu_{n} (I_{w_{k}})} = \frac{\m (I_{w_{k}})} {\nu_{n} (I_{w_{k}})}\quad \hbox{for $\nu_{n}$-a.e. $x\in I_{w_{k}}$}.
$$
Now we apply Theorem 5.7 in~\cite{V} to our case.

 \medskip
 \begin{lemma}[Martingale Convergence Theorem]~\label{mct}
 Since $X\in L_{1}(\nu_{n})$, the martingale $\{X_{k}=E(X|\B_{k})\}_{k=1}^{\infty}$ converges to $X$ for almost all points with respect to $\nu_{n}$ as $k$ goes to $\infty$.  
 \end{lemma}

\medskip
\medskip

\noindent {\bf Acknowledgment:} I would like to thank Professors Huyi Hu, Yun Yang, Wenbo Sun, Jiayu Chen, and my student Yuan (Jessica) Liu for listening, having helpful conversations, and pointing out 
some typos on this work. 
This work is supported partially by a PSC-CUNY Enhanced award (66680-54), Simons Foundation awards (523341 and
942077), and an NSFC grant (12271185).

 \end{document}